\documentclass[12pt]{amsart}
\numberwithin{equation}{section}
\usepackage{amsmath,amsthm,amsfonts,amscd,eucal}

\usepackage{xypic}
\usepackage{graphicx}

\usepackage{amssymb}
\def\bc{{\mathbb C}}

\def\bn{{\mathbb N}}

\def\br{{\mathbb R}}

\def\b{\beta}

\def\eps{\varepsilon}

\newtheorem{thm}{Theorem}[section]
\newtheorem{lem}[thm]{Lemma}

\newtheorem{prop}[thm]{Proposition}

\theoremstyle{definition}

\def\di{{\rm d}}

\begin{document}

\title[weighted Cesaro averages]
{Limits of some weighted Cesaro averages}
\author{Vitonofrio Crismale}
\address{Vitonofrio Crismale\\
Dipartimento di Matematica\\
Universit\`{a} degli studi di Bari\\
Via E. Orabona, 4, 70125 Bari, Italy}
\email{\texttt{vitonofrio.crismale@uniba.it}}
\author{Francesco Fidaleo}
\address{Francesco Fidaleo\\
Dipartimento di Matematica\\
Universit\`{a} degli studi di Roma Tor Vergata\\
Via della Ricerca Scientifica 1, Roma 00133, Italy} \email{{\tt
fidaleo@mat.uniroma2.it}}
\author{Yun Gang Lu}
\address{Yun Gang Lu\\
Dipartimento di Matematica\\
Universit\`{a} degli studi di Bari\\
Via E. Orabona, 4, 70125 Bari, Italy}
\email{\texttt{yungang.lu@uniba.it}}
\date{\today}

\begin{abstract}
We investigate the existence of the limit of some high order weighted Cesaro averages.

\end{abstract}

\maketitle

\section{introduction}

\label{sec1}

Motivated by potential applications to several branches of the mathematics, we study the possible convergence of high order weighted Cesaro means of the type
\begin{equation}
\label{cesczfi1}
\frac{1}{n^p}\sum_{k=1}^nb_kf(k/n)\,,
\end{equation}
where $p>0$ and $f:(0,1]\to\br$, provided $(b_k)_{k\in\bn}\subset\bc$ is a $p$-mean convergent sequence:
$$
\lim_n\frac{1}{n^p}\sum_{k=1}^nb_k=b\in\bc\,.
$$
Averages like those in \eqref{cesczfi1} naturally appear in Ergodic Theory. They also play a role in Probability, for example in the investigation of the central limit (see e.g. \cite{F}), as well as in Infinite Dimensional Analysis in managing the so-called L\'evy Laplacian (cf. \cite{L}) and exotic, i.e. high order ones, see e.g. \cite{AJS1} and the references cited therein. Cesaro averages as above might find natural applications also in Harmonic Analysis, Linear Algebra and Matrix Theory, Numerical Analysis, Number Theory and
in other sectors of pure and applied mathematics.

The convergence of the mean in \eqref{cesczfi1} depends on the conditions imposed on the function $f$, which are listed in our main result in Section \ref{sec1bis}. For example, we get convergence for the simple cases
$$
f(x)=x^q\,, \,\,f(x)=(1-x)^q\,, \quad q>0\,,
$$
which leads to the results in Section \ref{seca} concerning averages of multi-indices sequences.

The weighted averages of multi-indices sequences
appear in managing some quantum central limit theorems, when the sequence of mean covariances is not constant but at least convergent, and an order structure on some indices affects the value of the so-called mixed moments. Indeed, Propositions \ref{le-kro01} and \ref{le-kro02}, which quite surprisingly lead to results which cannot be reflected, may be naturally exploited in Anti-Monotone and Monotone cases (see e.g \cite{CFL, CFL2, Mur}). In order to get a flavour of the several kinds of mixed moments naturally emerging in Quantum Probability and the associated problem of their computation, the reader is referred to \cite{ACL, AHO, BS, CrLu} and the references cited therein.

The last section is devoted to counterexamples which explain that all the conditions imposed on our results are essentially optimal.

We end by noticing that particular cases of averages considered here appear in Section 5 of \cite{AJS} (see also \cite{AJS1}), where also several continuous versions of averages are investigated.

\section{limits of weighted cesaro means}
\label{sec1bis}

In the present note we suppose that the set of natural numbers does not contain 0:
$$
\bn:=\{1,2,\dots,n,\dots\}\,.
$$

We start with some elementary notations by denoting for each function $f:(0,1]\to\br$, a sequence ${\bf b}:=(b_n)_{n\in\mathbb{N}}\subset\bc$, and finally $p\in(0,+\infty)$,
$$
M_{{\bf b},f;p}(n):=\frac{1}{n^p}\sum_{k=1}^nb_kf(k/n)
$$
some useful high order weighted Cesaro means. For any sequence ${\bf b}$, by $|{\bf b}|$ we denote the sequence $(|b_n|)_{n\in\mathbb{N}}$.
A sequence ${\bf b}$ is said to be {\it $p$-mean convergent} if the sequence $(M_{{\bf b},1;p}(n))_{n\in\bn}$ of its Cesaro $p$-averages is convergent, where 1 stands for the constant function $f=1$ identically. When $p=1$, we recover the usual setting concerning the arithmetic means. It is easy to show that,
if ${\bf b}$ is $p$-mean convergent then $b_n=o(n^p)$ for $n\to+\infty$.

Let $f:(0,1]\to\br$ be a monotone function. Define on $(0,1]$ the possible infinite Borel measure $|\di f|$ induced by the Stieltjes integral with respect to $f$ if it is increasing, of by $-f$ if
$f$ is decreasing, see e.g. \cite{R}, Section 12.3.

The following result is useful in the sequel:
\begin{lem}
\label{le-si}
Let  ${\bf a}$ and ${\bf b}$ be convergent and $p$-mean convergent sequences with $\lim_na_n=a$ and $\lim_nM_{{\bf b},1;p}(n)=b$, respectively. Suppose that
\begin{equation}
\label{mcza}
M_{|{\bf b}|,1;p}(n)\leq B\,,\qquad n\in\bn\,,
\end{equation}
then the product sequence ${\bf ab}$ is $p$-mean convergent with
$$
\lim_n M_{{\bf ab},1;p}(n)=ab\,.
$$
\end{lem}
\begin{proof}
Fix $\eps>0$ and choose $l_0$ such that $n>l_0$ implies $|a_n-a|<\eps$. We get for $n>l_0$,
\begin{align*}
\bigg|M_{{\bf ab},1;p}(n)-ab\bigg|&\leq\bigg(\frac{l_0}{n}\bigg)^p\bigg(\bigg|M_{{\bf ab},1;p}(l_0)\bigg|+\bigg|a M_{{\bf b},1;p}(l_0)\bigg|\bigg)\\
&+\bigg|a \bigg(M_{{\bf b},1;p}(n)-b\bigg)\bigg| +\eps B\,.
\end{align*}
We then have
$$
\limsup_n\bigg|M_{{\bf ab},1;p}(n)-ab\bigg|\leq \eps B\,,
$$
which leads to the assertion being $\eps$ arbitrary.
\end{proof}

Here, there is our main result:
\begin{thm}
\label{le-kro00}
Fix a $p$-mean convergent sequence ${\bf b}$ with $\lim_nM_{{\bf b},1;p}(n)=b$, and a monotone function $f:(0,1]\to\br$ such that $f\in L^1((0,1],x^{p-1}\di x)$ and $x^p\in L^1((0,1],|\di f|)$.
Then
\begin{equation*}
\lim_{n}M_{{\bf b},f;p}(n)=bp\int_0^1x^{p-1}f(x)\di x\,.
\end{equation*}
\end{thm}
\begin{proof}
We can suppose, without loosing generality, that $f$ is decreasing by passing possibly to the opposite function, and positive by possibly adding a constant. Under the last hypotheses, for each
$\eps>0$ there exists $n_0$ such that, if $n>n_0$
$$
0\leq\sum_{k=1}^{[\frac{n}{n_0}]}f\bigg(\frac{k}{n}\bigg)\bigg[\bigg(\frac{k}{n}\bigg)^p-\bigg(\frac{k-1}{n}\bigg)^p\bigg]
\leq p\int_0^{1/n_0}x^{p-1}f(x)\di x\leq\eps\,,
$$
where $[x]$ is the unique integer such that $[x]\leq x< [x]+1$ for any arbitrary real $x$.
We then argue that
\begin{align*}
0\leq&p\int_0^{1}x^{p-1}f(x)\di x-\sum_{k=1}^nf\bigg(\frac{k}{n}\bigg)\bigg[\bigg(\frac{k}{n}\bigg)^p-\bigg(\frac{k-1}{n}\bigg)^p\bigg]\\
\leq& p\int_0^{1/n_0}x^{p-1}f(x)\di x+\sum_{k=1}^{[\frac{n}{n_0}]}f\bigg(\frac{k}{n}\bigg)\bigg[\bigg(\frac{k}{n}\bigg)^p-\bigg(\frac{k-1}{n}\bigg)^p\bigg]\\
+&\bigg\{p\int_{1/n_0}^1x^{p-1}f(x)\di x-\sum_{k=[\frac{n}{n_0}]+1}^n f\bigg(\frac{k}{n}\bigg)\bigg[\bigg(\frac{k}{n}\bigg)^p-\bigg(\frac{k-1}{n}\bigg)^p\bigg]\bigg\}\\
\leq&2\eps+\bigg\{p\int_{1/n_0}^1x^{p-1}f(x)\di x-\sum_{k=[\frac{n}{n_0}]+1}^n f\bigg(\frac{k}{n}\bigg)\bigg[\bigg(\frac{k}{n}\bigg)^p-\bigg(\frac{k-1}{n}\bigg)^p\bigg]\bigg\}\\
\to&2\eps
\end{align*}
for $n\to+\infty$, since one recognises the last term as the Riemann-Stieltjes sum of
$$
\int_0^1f(x)\di x^p=p\int_0^1 f(x)x^{p-1}\di x\,.
$$
As $\eps>0$ is arbitrary, we conclude that
\begin{equation}
\label{rsm}
\lim_n\sum_{k=1}^nf\bigg(\frac{k}{n}\bigg)\bigg[\bigg(\frac{k}{n}\bigg)^p-\bigg(\frac{k-1}{n}\bigg)^p\bigg]=p\int_0^{1}x^{p-1}f(x)\di x\,.
\end{equation}
With
$$
c_n:=M_{{\bf b},1;p}(n)-b\,,\quad n\in\bn\,,
$$
we get
\begin{align*}
&M_{{\bf b},f;p}(n)=c_nf(1)+b\sum_{k=1}^nf\bigg(\frac{k}{n}\bigg)\bigg[\bigg(\frac{k}{n}\bigg)^p-\bigg(\frac{k-1}{n}\bigg)^p\bigg]\\
+&\sum_{k=2}^nc_{k-1}\bigg(\frac{k-1}{n}\bigg)^p\bigg[f\bigg(\frac{k-1}{n}\bigg)-f\bigg(\frac{k}{n}\bigg)\bigg]\,.
\end{align*}
For each $\eps>0$, let $n_0$ such that $n>n_0$ implies $|c_n|<\eps$. Then for every $n$ sufficiently big,
\begin{align*}
&\bigg|\sum_{k=2}^nc_{k-1}\bigg(\frac{k-1}{n}\bigg)^p\bigg[f\bigg(\frac{k-1}{n}\bigg)-f\bigg(\frac{k}{n}\bigg)\bigg]\bigg|\\
\leq&\sum_{k=2}^{n_0+1}|c_{k-1}|\bigg(\frac{k-1}{n}\bigg)^p\bigg[f\bigg(\frac{k-1}{n}\bigg)-f\bigg(\frac{k}{n}\bigg)\bigg]\\
+&\sum_{k=n_0+2}^n|c_{k-1}|\bigg(\frac{k-1}{n}\bigg)^p\bigg[f\bigg(\frac{k-1}{n}\bigg)-f\bigg(\frac{k}{n}\bigg)\bigg]\\
<&\sup_n|c_n|\int_0^{\frac{n_0+1}{n}}\!x^p|\di f(x)|+\eps\int_0^1x^p|\di f(x)|\,,
\end{align*}
which goes to 0 as $n\to+\infty$, because $\eps>0$ is arbitrary. Collecting the last computation with \eqref{rsm}, we get the result.
\end{proof}

\section{some multi-dimensional cases}
\label{seca}

The present section is devoted to the investigation of some ergodic limits of multi-dimensional Cesaro averages which may appear in the study of Quantum Central Limit Theorems as those considered in \cite{CFL2}.
\begin{prop}
\label{le-kro01}
Let ${\bf b}$ be a $p$-mean convergent sequence satisfying \eqref{mcza} with $\lim_nM_{{\bf b},1;p}(n)=b$, and $(a_{k_{1},\ldots,k_{m}})_{k_{1},\ldots,k_{m}\in\mathbb{N}}\subset\bc$ a multi-indices sequence such that for $q>0$,
$$
\lim_{n}\frac{1}{n^{q}}\sum_{1\leq k_{1},\ldots,k_{m}\leq
n}a_{k_{1},\ldots,k_{m}}=a\,.
$$
Then
\begin{equation*}
\lim_{n}\frac{1}{n^{p+q}}\sum_{k=1}^nb_k\sum_{1\leq k_{1},\ldots,k_{m}\leq k}a_{k_{1},\ldots,k_{m}}=\frac{abp}{p+q}\,.
\end{equation*}
\end{prop}
\begin{proof}
Notice that
$$
\frac{1}{n^{p+q}}\sum_{k=1}^nb_k\sum_{1\leq k_{1},\ldots,k_{m}\leq k}a_{k_{1},\ldots,k_{m}}
=M_{{\bf ab},x^q;p}(n)\,,
$$
where
$$
a_k:=\frac1{k^q}\sum_{1\leq k_{1},\ldots,k_{m}\leq k}a_{k_{1},\ldots,k_{m}}\,,\qquad k\in\bn\,,
$$
defines the sequence ${\bf a}$ which is supposed to be convergent. The proof now follows from Lemma \ref{le-si} and Theorem \ref{le-kro00}.
\end{proof}
Recall that the Euler's Beta and Gamma functions are defined respectively as
$$
\b(z,t):=\int_0^1 x^{z-1}(1-x)^{t-1}\di x\,,\quad {\rm Re}(z), {\rm Re}(t)>0\,,
$$
$$
\Gamma(z):=\int_0^{+\infty}x^{z-1}e^{-x}\di x\,,\quad z\in\mathbb{C}\backslash\{0,-1,-2,\ldots\}\,.
$$
Such special functions are related by the celebrated identity
\begin{equation}
\label{eul}
\b(z,t)=\frac{\Gamma(z)\Gamma(t)}{\Gamma(z+t)}\,,
\end{equation}
see e.g. \cite{E}.

The functions above appear in the following result concerning the tail-average.
\begin{prop}
\label{le-kro02}
Let $(a_{k_{1},\ldots,k_{m}})_{k_{1},\ldots,k_{m}\in\mathbb{N}}$ and ${\bf b}$ be a multi-indices sequence and a sequence respectively, satisfying all the hypotheses of Proposition \ref{le-kro01}. If in addition,
\begin{equation}
\label{invtr}
a_{k_{1}-h,\ldots,k_{m}-h}=a_{k_{1},\ldots
,k_{m}}
\end{equation}
for any $k_{1},\ldots,k_{m}\in\mathbb{N}$ and $h<
\min\{k_{1},\ldots,k_{m}\}$,
then
\begin{equation}
\label{kro00b}
\lim_{n}\frac{1}{n^{p+q}}\sum_{k=1}^n b_k\sum_{k+1\leq k_{1},\ldots,k_{m}\leq n}a_{k_{1},\ldots
,k_{m}}=ab\frac{\Gamma(p+1)\Gamma(q+1)}{\Gamma(p+q+1)}\,.
\end{equation}
\end{prop}
\begin{proof}
Notice that \eqref{invtr} gives
$$
\sum_{k+1\leq k_{1},\ldots,k_{m}\leq n}a_{k_{1},\ldots,k_{m}}=\sum_{1\leq k_{1},\ldots,k_{m}\leq n-k}a_{k_{1},\ldots,k_{m}}
$$
and, consequently,
\begin{align*}
&\frac{1}{n^{p+q}}\sum_{k=1}^n b_k\sum_{k+1\leq k_{1},\ldots,k_{m}\leq n}a_{k_{1},\ldots,k_{m}}\\
=&\frac{1}{n^{p+q}}\sum_{k=1}^n b_k(n-k)^{q}\bigg[\frac{1}{(n-k)^{q}}
\sum_{1\leq k_{1},\ldots,k_{m}\leq n-k}a_{k_{1},\ldots,k_{m}}-a\bigg]\\
+&\frac{a}{n^{p+q}}\sum_{k=1}^n b_k(n-k)^{q}\,.
\end{align*}
From Proposition \ref{le-kro00}, one has
\begin{align*}
\lim_{n}\frac{1}{n^{p+q}}\sum_{k=1}^n b_k(n-k)^{q}&=bp\int_0^1 x^{p-1}(1-x)^q \di x \\
&=b\frac{\Gamma(p+1)\Gamma(q+1)}{\Gamma(p+q+1)},
\end{align*}
the last equality coming from \eqref{eul} and $\Gamma(z+1)=z\Gamma(z)$.

The thesis then follows once one shows
\begin{equation*}
\frac{1}{n^{p+q}}\sum_{k=1}^n b_k(n-k)^{q}\bigg[\frac{1}{(n-k)^{q}}
\sum_{1\leq k_{1},\ldots,k_{m}\leq n-k}a_{k_{1},\ldots,k_{m}}-a\bigg]
\end{equation*}
is infinitesimal for $n\rightarrow\infty$.
Indeed, since for any $\eps>0,$ there is $l_{0}\in\mathbb{N}$ such that for any
$h\geq l_{0}$
$$
\bigg|\frac{1}{h^{q}}\sum_{1\leq k_{1},\ldots,k_{m}\leq h}a_{k_{1},\ldots,k_{m}}-a\bigg|\leq\eps,
$$
one has for each $k=1,\ldots, n-l_0$,
$$
\bigg|\frac{1}{(n-k)^{q}}
\sum_{1\leq k_{1},\ldots,k_{m}\leq n-k}a_{k_{1},\ldots,k_{m}}-a\bigg|\leq\eps.
$$
Thus, denoting by $M>0$ a uniform bound for the sequence of the multiple of Cesaro means of $(a_{k_{1},\ldots,k_{m}})$, by \eqref{mcza} one finds
\begin{align*}
&\bigg|\frac{1}{n^{p+q}}\sum_{k=1}^n b_k(n-k)^{q}\bigg[\frac{1}{(n-k)^{q}}
\sum_{1\leq k_{1},\ldots,k_{m}\leq n-k}a_{k_{1},\ldots,k_{m}}-a\bigg]\bigg| \\
\leq&\frac{ \eps}{n^p}\sum_{k=1}^{n-l_0} |b_k|\bigg(\frac{n-k}{n}\bigg)^{q}\\
+&\bigg|\frac{1}{n^{p}}\sum_{k=n-l_0+1}^{n}b_k \bigg(\frac{n-k}{n}\bigg)^{q}\bigg[\frac{1}{(n-k)^{q}}
\sum_{1\leq k_{1},\ldots,k_{m}\leq n-k}a_{k_{1},\ldots,k_{m}}-a\bigg]\bigg|\\
\leq&\bigg[\eps +2M\bigg(\frac{l_0}{n}\bigg)^{q}\bigg]B\,.
\end{align*}
The proof is achieved as $\eps$ is arbitrary.
\end{proof}

\section{some counterexamples}

We end the present note by showing some counterexamples concerning the average-convergence of sequences.

We start by noticing that in Theorem \ref{le-kro00},
the case with ${\bf b}$ identically equal to 1 and $p=1$ corresponds simply to ask whether the sequence of the Riemann sums of a $L^1$-function $f$, made partitioning the interval $[0,1]$ in $n$ subintervals of uniform length $1/n$ and Riemann integrable on all the subintervals $[\eps, 1]$, converges to the integral of $f$. The following simple counterexample (which can be easily modified to achieve the continuous case)
$$
f=\sum_{n=1}^{+\infty}n^2\chi_{\{1/n\}}
$$
tells us that it is not always the case, even if one imposes mild natural conditions on $f$.

Now we pass to see that the convergence of $\frac1{n}\sum_{k=1}^n|b_k|$ does not imply that ${\bf b}$ is mean-convergent. Let ${\bf b}$ be the sequence defined as
$$
{\bf b}:=\overbrace{1}^{2^0}\,,\overbrace{-1,-1}^{2^1}\,,\overbrace{1,1,1,1}^{2^2}\,,\overbrace{-1,\dots,-1}^{2^3}\,,\dots\,\,.
$$
Define, for each integer $n$,
$$
m_n:=2\cdot4^n-1\,,\quad h_n:=4^{n+1}-1\,.
$$
On one hand, it is easy to check that
$$
\frac1{n}\sum_{k=1}^n|b_k|=1.
$$
On the other hand, for the subsequences indexed by $m_n$ and $h_n$ respectively, one finds
$$
M_{{\bf b},1;1}(m_n)=\frac1m_n\bigg(\sum_{k=0}^n 2^{2k}-\frac12 \sum_{k=1}^n 2^{2k}\bigg)=\frac13\,,
$$
and
$$
M_{{\bf b},1;1}(h_n)=\frac1h_n \bigg(\sum_{k=0}^n 2^{2k}-\frac12 \sum_{k=1}^{n+1} 2^{2k}\bigg)=-\frac13\,.
$$
What follows is a simple counterexample for the general failure of Lemma \ref{le-si} if condition \eqref{mcza} is not satisfied.
Let ${\bf b}=(b_k)_{k\in \mathbb{N}}$ and ${\bf a}=(a_k)_{k\in \mathbb{N}}$ be defined as follows:
\begin{align*}
&b_{2n-1}:=-\sqrt{2n}\,,\quad b_{2n}:=1+\sqrt{2n}\,,\quad n\in\bn\,,\\
&a_{2n-1}:=-\frac{1}{\sqrt{2n}}\,,\quad a_{2n}:=\frac{1}{\sqrt{2n}}\,,\quad n\in\bn\,.
\end{align*}
Then $a=\lim_n a_n=0$, and $b=\lim_nM_{{\bf b},1;1}(n)=\frac{1}{2}$. Furthermore, as $n\to+\infty$, first
$$
\frac{1}{2n}\sum_{k=1}^{2n}|b_k|=\frac{1}{2}+\frac{1}{n}\sum_{k=1}^{n}\sqrt{2k}\to+\infty\,,
$$
and second
$$
M_{{\bf ab},1;1}(2n)=1+\frac{1}{2n}\sum_{k=1}^{n}\frac1{\sqrt{2k}}\to1>0=ab\,.
$$
Finally, one can wonder if
\eqref{kro00b} holds true
under all the assumptions of Proposition \ref{le-kro01} but \eqref{invtr}.
The answer is negative as the following example shows for the case $p=1$, $m=2$, and $q=m$.
Indeed, take
$$
b_k=1\,,\quad a_{k_1,k_2}=(\sqrt
{k_1}-\sqrt{k_1-1})\sqrt{k_2}\,,\quad k,k_1,k_2\in\mathbb{N}\,.
$$
Then $b=1$ and $a=\frac{2}{3}$ as
\begin{align*}
&\lim_{n}\frac{1}{n^{2}}\sum_{1\leq k_1,k_2\leq n}a_{k_1,k_2}=\lim_{n}\frac{1}{n^{2}}\sum_{1\leq k_1,k_2\leq n}(  \sqrt{k_1}-\sqrt{k_1-1})\sqrt{k_2}\\
=&\lim_{n}\frac{1}{n^{2}}\sqrt{n}\sum_{k_2=1}^{n}\sqrt{k_2}=\lim_{n}\frac{1}{n}\sum_{k_2=1}^{n}\sqrt{\frac{k_2}{n}}
=\int_{0}^{1}x^{\frac{1}{2}}dx=\frac{2}{3}\,.
\end{align*}
Computing the left hand side of \eqref{kro00b}, we get
\begin{align*}
&\lim_{n}\frac{1}{n^{3}}\sum_{k=1}^n b_k\sum_{k+1\leq k_{1},k_{2}\leq n} a_{k_1,k_2}=\lim_{n}\frac{1}{n^{3}}\sum_{k=1}^{n}(\sqrt{n}-\sqrt{k})\sum_{k_2=k+1}^{n}\sqrt{k_2}\\
=&\lim_{n}\frac{1}{n^{2}}\sum_{k=1}^{n}\bigg(1-\sqrt{\frac{k}{n}}\bigg)\sum_{k_2=k+1}^{n}\sqrt{\frac{k_2}{n}}
=\int_{0}^{1}\di x(1-\sqrt{x})\int_{x}^{1}\di y\sqrt{y}\\
=&\frac{4}{15}ab\neq\frac{ab}3\,.
\end{align*}

\section*{Note added in proof}

The authors are grateful to O. Kouba who has drawn their attention to Theorem 1 in his note \cite{K} while the present article was in press. The statement of such a theorem is the same as our 
Theorem \ref{le-kro00}, provided that the involved function $f$ and the sequence $(b_n)_{n\in\bn}$ are uniformly continuous on $(0,1]$ and positive, respectively. By using Weierstrass' Density Theorem as in \cite{K}, the former is a corollary of the latter, and can be extended to general  $p$-mean convergent complex-valued sequences $(b_n)_{n\in\bn}$, provided that the sequence of their moduli 
$(|b_n|)_{n\in\bn}$ satisfies \eqref{mcza}.

\section*{Acknowledgements}

The authors have been partially supported by Italian INDAM--GNAMPA. They kindly acknowledge R. Peirone for some fruitful suggestions.

\end{document}